\documentclass[11pt,a4paper,oneside]{amsart}
\usepackage{url}
\theoremstyle{plain}
\newtheorem{theorem}{Theorem}[section]
\newtheorem{lemma}[theorem]{Lemma}
\theoremstyle{definition}
\newtheorem{definition}[theorem]{Definition}
\newtheorem{example}[theorem]{Example}

\newtheorem{corollary}[theorem]{Corollary}
\newtheorem{remark}[theorem]{Remark}
\newtheorem{fact}[theorem]{Fact}
\newtheorem{notation}[theorem]{notation} 
\newtheorem{Notation and Remark}[theorem]{Notation and Remark}
\newtheorem{proposition}[theorem]{Proposition}
\numberwithin{equation}{section}
\newcommand{\Stab}[1][R]{Stab_{\alfa}(#1)}
\usepackage{amssymb}

\usepackage{amsmath}
\usepackage{fullpage}
\DeclareMathOperator{\alfa}{\alpha}

\newcommand{\PrPol}[1][R]{\mathcal{P}(#1)}
\newcommand{\PolFun}[1][R]{\mathcal{F}(#1)}
\newcommand{\UPF}[1][R]{\mathcal{F}(#1)^{\times}}
\newcommand{\Aut}[1][G]{Aut(#1)}
 \usepackage{comment}
 \usepackage[inline]{enumitem} 
 \usepackage{hyperref}
\begin{document}

\openup.25em
\title[AMS Article Template]{On the group of unit-valued polynomial functions }
% \title[short text for running head]{full title}
%    Only \author and \address are required; other information is
%    optional.  Remove any unused author tags.

%    author one information
% \author[short version for running head]{name for top of paper}
\author{Amr Ali Al-Maktry}
\address{\hspace{-12pt}Department of Analysis and Number Theory (5010) \\ 
	Technische Universit\"at Graz \\
Kopernikosgasse 24/II  \\
	8010 Graz, Austria}
\curraddr{}
\email{almaktry@math.tugraz.at}
\thanks{}

%    author two information

%    \subjclass is required.
\date{}

\dedicatory{}

\maketitle

\begin{abstract}
	Let $R$ be a finite commutative ring  with  $1\ne 0$. The set $\mathcal{F}(R)$ 
	of polynomial functions on $R$ is a finite commutative ring with    pointwise operations.
	Its group of units   $\mathcal{F}(R)^\times$ is  just the set  of all unit-valued polynomial functions,
	that is the set of polynomial functions which map $R$ into its group of units.		  
	We show  that $\mathcal{P}_R(R[x]/(x^2))$ the group of polynomial permutations on the ring  
	$R[x]/(x^2)$,  consisting of permutations  represented by polynomials  over $R$,  
	is embedded in a semidirect product of $\mathcal{F}(R)^\times$ by $\mathcal{P}(R)$ the 
	group of polynomial permutations on $R$.  In particular, when $R=\mathbb{F}_q$,  we prove 
	that   $\mathcal{P}_{\mathbb{F}_q}(\mathbb{F}_q[x]/(x^2))\cong 
	\mathcal{P}(\mathbb{F}_q) \ltimes_\theta \mathcal{F}(\mathbb{F}_q)^\times$.
	Furthermore,  we count unit-valued polynomial functions $\pmod{p^n}$ and obtain  canonical
	representations for these functions.
	% \PACS{PACS code1 \and PACS code2 \and more}
	% \subclass{MSC code1 \and MSC code2 \and more}
\end{abstract}

\section{Introduction}
\label{intro}

Throughout this paper   $R$ is a finite commutative ring with unity $1\ne 0$. We denote by
$R^{\times}$  the group of units of $R$.  A function $F\colon R\longrightarrow R$ is called 
a polynomial function on $R$ if there exists a polynomial $f\in R[x]$ such that 
$F(r)=f(r)$ for each $r\in R$. In this case, we say that $f$ induces (represents) $F$ or $F$
is induced (represented) by $f$. If $F$ is a bijection,  we say that $F$ is a \emph{polynomial
	permutation} on $R$ and $f$  is a \emph{permutation polynomial} on $R$ (or $f$ permutes $R$). 
When $F$ is the constant zero, $f$ is called a null polynomial on $R$ or shortly, null on $R$. 
The set of all null polynomials  is an ideal of $R[x]$, which we denote  by $N_R$.

It is evident that the set $\PolFun $ of  all polynomial functions on $R$  is a monoid 
with composition of functions. Its group of invertible elements $\PrPol$ consists of polynomial 
permutations on $R$, and it is called the group of polynomial permutations on $R$. 
Further, $\PolFun $ is a ring with addition and multiplications defined pointwise. 

In this paper we are interested in the group of units of the pointwise ring structure on
$\PolFun$, which we denote by $\UPF$. We show that %this group plays a role in describing a
there is a relation between the group $\UPF$ and the group of polynomial permutations 
on $R[x]/(x^2)$, which is represented by polynomials with coefficients in $R$ only. 
Moreover, when $R=\mathbb{Z}_{p^n}$ the ring of integers modulo $p^n$ we find  the order 
of $\UPF[\mathbb{Z}_{p^n}]$ and give   canonical representations for its elements.

\section{Preliminaries}\label{Prel}

In this section, we introduce the concepts and notations used frequently in the paper. 
Throughout this paper whenever $A$ is a ring and  $f\in A[x]$,
let $[f]_A$ denote the polynomial function on $A$ represented by $f$. That is,
if $F$ is the function induced by $f$ on $R$, then $[f]_A=F$.

\begin{definition}\label{1}
	A polynomial $f\in R[x]$ is called  a \emph{unit-valued polynomial} if $f(r)\in R^{\times}$ 
	for each $r\in R$ \cite{UVPdf}. In this case $[f]_R$, the function induced by $f$ on $R$, 
	is called a \emph{unit-valued polynomial function}.
\end{definition}
Throughout this paper for every $f\in R[x]$, let $f'$ denote its first formal derivative.

Unit-valued polynomials and unit-valued polynomial functions have been appeared in the literature
within studying other structures. For example, Loper~\cite{UVPdf} has used unit-valued polynomials
for examining non-$D$-rings. One other example is seen in the criteria of permutation polynomials on
finite local rings. 
To illustrate  this we recall the well known fact.
\begin{fact}\cite[Theorem~3]{Necha} 
	Let $R$ be a local ring with maximal ideal $M$, and let $f\in R[x]$. Then $f$ is a permutation
	polynomial on $R$ if and only if the following conditions hold
	\begin{enumerate}
		\item  $\bar{f}$ is a permutation polynomial on the residue field $R/M$, where $\bar{f}$ denotes the reduction
		of $f$ modulo $M$;
		\item $f'(a)\ne 0 \mod M$ for every $a\in M$.
	\end{enumerate}
\end{fact}
Indeed, the second condition of the previous fact 
requires that $f'$ to be a unit-valued polynomial  on $R $ (equivalently $\bar{f}'$ is  a unit-valued polynomial
on $R/M$), that is $[f']_R$ is a unit-valued polynomial function.

To prove our first fact about unit-valued polynomial functions we need the following lemma, 
which is a special case of a general property proved in \cite{Regunit}. 

\begin{lemma}\label{reg}
	Let $R$ be a finite commutative ring with  $1\ne 0$. Then every regular element 
	(i.e., element which is not a zero divisor) is invertible.
\end{lemma}

\begin{proof}
	Let $a\in R$ be a regular element. Define a function $\phi_a\colon R \longrightarrow R$ by
	$\phi_a(r) =ar$ for each $r\in R$. We claim that $\phi_a$ is injective, and hence 
	it is surjective since $R$ is finite. Because, if there exist $r_1,r_2\in R$ such that
	$r_1\ne r_2$ and  $\phi_a(r_1) =ar_1=ar_2= \phi_a(r_2)$, then $a(r_1-r_2)=0$. 
	But this implies that $a$ is a zero divisor, which is a contradiction. 
	Thus  $\phi_a$ is bijective, and there exists a unique element $r$ 
	such that $1=\phi_a(r)=ar$.
\end{proof}

From now on, let $"\cdot"$ denote the pointwise multiplication of functions.

\begin{fact}\label{2}
	Let $\PolFun$ be the set of polynomial functions on $R$. Then $\PolFun$ is
	a finite commutative ring with nonzero unity, where addition and multiplication
	are defined pointwise. Moreover, $\PolFun^{\times}$ is an abelian group and;
	\[\PolFun^{\times}=\{F\in\PolFun: F \text{ is a unit-valued polynomial function}\}.\]
\end{fact}

\begin{proof}
	We leave to the reader t check that $\PolFun$ forms a finite commutative ring under pointwise
	operations with with unity is the constant $1$ on $R$, which we denote by $1_{\PolFun}$.
	
	Moreover,	since $\PolFun$ is commutative ring, it follows that $\PolFun^{\times}$ is an abelian group.
	Now, it is easy to see  that every unit-valued polynomial function is regular,
	and hence it is invertible by  Lemma~\ref{reg}.
	Thus $\UPF$ contains every unit-valued polynomial functions.
	
	For the other inclusion,
	let $F\in \UPF$. Then there exists $F^{-1}\in \UPF$ such that  $F\cdot F^{-1}=1_{\PolFun}$. 
	Hence $F(r)F^{-1}(r)=1$ for each $r\in R$, whence $F(r)\in R^{\times}$ for each $r\in R$. 
	Therefore $F$ is a unit-valued polynomial function by Definition~\ref{1}. 
	This completes the proof.
\end{proof}
\begin{remark}
	When $R$ is an infinite commutative ring, still true that $\PolFun$ is a commutative ring (infinite) and every element
	of $\UPF$ is a unit-valued polynomial function.
	However, $\UPF$ can be properly contained in the set of all unit-valued polynomial functions.  
\end{remark}
The following example illustrates the previous remark. 
\begin{example}
	Let $R=\{\frac{a}{b}: a,b\in \mathbb{Z}, b\ne 0\text{ and } 2\nmid b\}$. Then the polynomial
	$f=1+2x$ is a unit-valued polynomial on $R$, and $F=[f]_R$ is a unit-valued polynomial function.
	We claim that $F$ has no inverse in $\PolFun$. Assume, on the contrary, that $F$ is invertible, that is 
	there exists $F_1\in \PolFun$ such that $F\cdot F_1=1_{\PolFun}$, i.e., $F(r)F_1(r)=1$ for every $r\in R$.
	Now, since $F_1\in\PolFun$, there exists $f_1\in R[x]$ such that $F_1=[f_1]_R$.
	Then the polynomial $h(x)=(1+2x)f_1(x) -1$  is of positive degree. Further, $h$ has infinite roots
	in $R$ since $h(r)=F(r)F_1(r)-1=0$ for every $r\in R$, which contradicts the fundamental theorem of algebra. 
\end{example}
\begin{remark}
	For a commutative $R$, the ring $R[x]/(x^2)$ is called the ring of dual numbers over $R$. 
	This ring can be viewed as the ring  $R[\alpha]=\{a+b\alpha: a,b \in R, \alpha^2=0\}$,
	where $\alpha$ denotes the element $x+(x^2)$. In this construction $R$ is a subring of 
	$R[\alpha]$. Therefore every polynomial $g\in R[x]$ induces two functions on $R[\alpha]$
	and $R$ respectively, namely $[g]_{R[\alpha]}$ and its restriction (on $R$) $[g]_R$.
\end{remark}

The following fact about the polynomials of $R[\alpha]$ can be proved easily.
%By the construction of the ring $R[x]/(x)$ it is well known the following about their polynomials.

\begin{fact}\label{3}
	Let $R$ be a commutative ring, and %and $\Ralfa$ the ring of dual numbers   over $R$.
	$a,b\in R$.
	\begin{enumerate}
		\item
		Let $g\in R[x]$. Then
		$g(a+b\alfa)=g(a)+ bg'(a)\alfa$.
		
		\item
		Let $g\in R[\alfa][x]$. Then there exist unique  $g_1, g_2 \in R[x]$
		such that $g =g_1 +g_2 \alfa$  and
		
		$g(a+b\alfa)=g_1(a)+ (bg_1'(a)+g_2(a))\alfa$.\label{3-b}
		
	\end{enumerate}
\end{fact}

\begin{fact}\label{4}
	Let $g\in R[x]$. Then $g$ is a null polynomial on $R$ if and only if $g\alpha$ is
	a null polynomial on $R[\alpha]$.
\end{fact}

\begin{proof}
	$(\Leftarrow)$ Immediately since $R$ is a subring of $R[\alpha]$.
	
	$(\Rightarrow)$ Let $a,b \in R$. Then, by Fact~\ref{3}-(\ref{3-b}),
	$$g(a+b\alpha)\alpha=(g(a)+g'(a)b\alpha)\alpha=g(a)\alpha+0=0\alpha=0.$$
\end{proof}

Recall from the introduction that $\PrPol[{R[\alpha]}]$ denotes the group of 
polynomial permutations on ${R[\alpha]}$. It is apparent that $\PrPol[{R[\alpha]}]$ 
is a finite set, being a subset of a finite set (i.e., a subset of the set of polynomial 
functions on  $R[\alpha]$).

\begin{definition}\label{5}
	Let 
	$\mathcal{P}_R(R[\alpha])=\{F\in \PrPol[{R[\alpha]}] : F=[f]_{R[\alpha]},\ f\in R[x]\}$.
\end{definition}

From now on, let $"\circ"$ denote the composition of functions (polynomials) and let $id_R$
denote the identity function on $R$.

\begin{fact}
	The set $\mathcal{P}_R(R[\alpha])$ is a subgroup of $\PrPol[{R[\alpha]}]$.
\end{fact}

\begin{proof}
	Evidently, $id_{R[\alpha]}\in \mathcal{P}_R(R[\alpha])$. Since $\mathcal{P}_R(R[\alpha])$ 
	is finite, it suffices to show that $\mathcal{P}_R(R[\alpha])$ is closed under composition.
	So if $F_1,F_2\in \mathcal{P}_R(R[\alpha])$, then $F_1,F_2$ are induced by 
	$f_1,f_2\in R[x]$ respectively by Definition~\ref{5}. Now 
	$F_1\circ F_2\in \PrPol[{R[\alpha]}]$ (being a composition of polynomial permutations 
	on $R[\alpha]$). But, it is evident that $F_1\circ F_2=[f_1\circ f_2]_{R[\alpha]}$ 
	with  $f_1\circ f_2\in R[x]$. Hence $F_1\circ F_2\in \mathcal{P}_R(R[\alpha])$ 
	by Definition~\ref{5}.
	%It is evident that $\mathcal{P}_R(R[\alpha])$ is a subgroup of $\PrPol[{R[\alpha]}]$.
\end{proof}

\begin{remark}
	Let $f,g\in R[x]$. Then their composition $g\circ f$ induces a function on $R$, 
	which is the composition of the functions induced by  $f$ and $g$ on $R$. 
	Similarly, $f+g$ and $fg$ induce two functions  on $R$, namely the pointwise
	addition and multiplication  respectively of the functions induced by $f$ and $g$. 
	In terms of our notations this equivalent to the following:	
	\begin{enumerate}
		\item $[f\circ g]_R=[f]_R\circ [g]_R$;
		\item $[f+ g]_R=[f]_R+ [g]_R$;
		\item $[f g]_R=[f]_R\cdot [g]_R$.
	\end{enumerate}
	
	The above  equalities appear periodically  in our arguments in the next sections. 
\end{remark}

\section{The embedding of the group $\mathcal{P}_R(R[\alpha])$ in the Group $\PrPol \ltimes_\theta  \PolFun^\times$}\label{embed}

In this section we show that the group $\mathcal{P}_R(R[\alpha])$, which consists of
permutations represented by polynomials from $R[x]$,  is embedded in a semidirect 
product of the group of unit-valued polynomial functions on $R$, $\PolFun^\times$ by
the group of polynomial permutations on $R$, $\PrPol$, via a homomorphism $\theta$ 
defined in Lemma~\ref{7-aut} below.

\begin{lemma}\label{05}
	Let $F,F_1\in \UPF$ and $G\in \PrPol$. Then the following hold:
	\begin{enumerate}
		\item $F\circ G\in\UPF$;
		\item $(F\cdot F_1)\circ G=(F\circ G)\cdot(F_1\circ G)$;
		\item if $F^{-1}$ is the inverse of $F$, then $F^{-1}\circ G$ is the inverse of $F\circ G$.
	\end{enumerate}
\end{lemma}

\begin{proof}
	(1) Obviously,   $F\circ G$ is a polynomial function. Further,
	$F(G(r))\in R^\times$ for every $r\in R$, and so $F\circ G\in \UPF$.
	
	(2) Put $F_2=F\cdot F_1$.  So  if $r\in R$ then, by the pointwise multiplication
	on $\UPF$,
	$$F_2\circ G(r)=F_2(G(r))=F(G(r))F_1(G(r))=(F\circ G(r))(F_1\circ G(r)).$$
	Hence $(F\cdot F_1)\circ G=(F\circ G)\cdot (F_1\circ G)$.
	
	(3) By (2),
	$$(F^{-1}\circ G)\cdot (F\circ G)=(F^{-1}\cdot F)\circ G=1_{\PolFun}\circ G=1_{\PolFun}.$$
\end{proof}

An expert reader should notice that Lemma~\ref{05} defines a group action of $G$ 
on $\UPF$ in which every element of $G$ induces a homomorphism on $\UPF$, and
what is coming now is a consequence of that. However, we do not refer to this action
explicitly to avoid recalling  additional materials. In fact, our arguments are 
elementary and depend on direct calculations.

\begin{lemma}\label{7-aut}
	Let  $G\in \PrPol$. Then the map $\theta_G\colon\UPF \longrightarrow \UPF$ defined 
	by $(F)\theta_G=F\circ G$, for all $F\in \UPF$, is an automorphism of $\UPF$. 
	Furthermore, the map 
	$\theta\colon \PrPol \longrightarrow \Aut[\UPF]$   
	defined by $(G)\theta=\theta_G$ is a homomorphism.       
\end{lemma}

\begin{proof}In view of
	Lemma~\ref{05}-(2) we need only to show that $\theta_G$ is a bijection.
	Let $F\in \UPF$. Then $F\circ G^{-1}\in \UPF$ by Lemma~\ref{05}, and we have that
	$$(F\circ G^{-1})\theta_G=(F\circ G^{-1})\circ G= F\circ(G^{-1}\circ G)=F\circ id_R=F.$$ 
	This shows that $\theta$ is a surjection, and hence it is a bijection since  $\UPF$ is finite.
	Therefore $\theta_G \in \Aut[\UPF]$.
	
	Moreover, if $\theta\colon \PrPol \longrightarrow \Aut[\UPF]$ is given by $(G)\theta=\theta_G$,
	then for every $G_1,G_2\in \PrPol$ and any $F\in \UPF$, 
	we have
	$$(F)\theta_{G_1\circ G_2}=F\circ(G_1\circ G_2)=(F\circ G_1)\circ G_2 =
	(F\circ G_1)\theta_{G_2}=((F)\theta_{ G_1})\theta_{G_2}=(F)\theta_{ G_1}\circ\theta_{G_2}.$$
	Hence $\theta_{ G_1\circ G_2}= \theta_{ G_1}\circ\theta_{G_2}$ 
	and $\theta$ is a homomorphism.
\end{proof}

\begin{Notation and Remark}\label{07}
	Let $H$ be the set of all pairs $(G,F)$, where $G\in \PrPol$ and $F\in \UPF$.
	We put 
	\[\overline{\PrPol}=\{( G,1_{\PolFun}): G\in \PrPol\}, \text{ and }
	\overline{\UPF}=\{( id_R,F): F\in \UPF\}.\] In the following proposition
	we use the homomorphism $\theta$ of Lemma~\ref{7-aut}  to define 
	a multiplication on $H$. Such a multiplication allows us to view  $H$ as 
	the semidirect of   $\overline{\UPF}$ by $\overline{\PrPol}$.
\end{Notation and Remark}

\begin{proposition}\label{semiconst}
	Assume the above notation. Define a multiplication on $H$ by
	$$(G_1,F_1)(G_2,F_2)=\big( G_1\circ G_2,(F_1)\theta_{G_2 }\cdot F_2 \big)=
	\big(G_1\circ G_2,(F_1\circ G_2)\cdot F_2\big).$$
	Then $H$ is a group containing $\overline{\PrPol},\overline{\UPF}$ as subgroups. 
	Furthermore, the following hold:
	\begin{enumerate}
		\item $H=\overline{\PrPol} \ \overline{\UPF}$;
		\item $\overline{\UPF}\lhd H$;
		\item $\overline{\PrPol} \cap \overline{\UPF}=\{(id_R,1_{\PolFun})\}$.
	\end{enumerate}
	
	That is, $H$ is the (internal) semidirect product of $\overline{\UPF}$ by $\overline{\PrPol}$.
\end{proposition}
The proof of Proposition~\ref{semiconst} depends essentially on Lemma~\ref{7-aut}, and it is just the justifications of the 
semidirect product properties (see for example \cite{fgroup,finitegroup2}).  
\begin{comment}
Therefore the multiplication  in $\UPF \rtimes_\theta\PrPol$ is defined by
$$(G_1,F_1)(G_2,F_2)=\big( G_1\circ G_2,(F_1\circ G_2)\cdot F_2 \big)$$
for every $F_1,F_2\in \UPF$,  $G_2, G_1\in \PrPol$, where $\cdot$ denotes  the pointwise multiplication on the group $\UPF$.
For the associative property one has\\
$(G_1,F_1)\big((G_2,F_2)(G_3,F_3)\big)=\big((G_1,F_1) (G_2,F_2)\big)(F_3,G_3)=\big(G_1\circ G_2 \circ G_3,(F_1\circ G_2\circ G_3)\cdot(F_2\circ G_3)\cdot F_3\big)$, where  $G_i\in \PrPol $,  $F_i \in \UPF$ for $i=1,2,3$.\\
We leave to the reader to check that $(id_R,1_{\PolFun})$ is the identity of $\UPF \rtimes_\theta\PrPol$, where $id_R$ is the identity function on $R$, and $1_{\PolFun}$ is the unit-valued polynomial function sending every element of $R$ into $1$.
Let $F_1,F_2\in \UPF$, $G_1,G_2\in \PrPol$  such that $F_1. F_2=F_2. F_1=1_{\PolFun}$, and $G_1\circ G_2=G_2\circ G_1=id_R$. Then simple calculations show that\\
$(F_1,G_1)(F_2\circ G_2,G_2)=(F_2\circ G_2,G_2)(F_1,G_1)=(id_R,1_{\PolFun})$.
\end{comment}

\begin{remark}\label{007}
	In fact, the subgroups $\overline{\PrPol}$ and $ \overline{\UPF}$ are isomorphic
	to ${\PrPol}$ and $ {\UPF}$ respectively. In this case $H$ is called the (external)
	semidirect product of $\UPF$ by $\PrPol$, and it is denoted by 
	$\PrPol \ltimes_\theta \PolFun^\times$ (see for example \cite{fgroup,finitegroup2}).  
	For this reason,	we simply use $ \PrPol \ltimes_\theta \PolFun^\times$ to mention  
	the semidirect product constructed in Proposition~\ref{semiconst}.
\end{remark}

Our next aim is to show that the group $\mathcal{P}_R(R[\alpha])$ defined in 
Definition~\ref{5} is embedded in    $ \PrPol \ltimes_\theta \PolFun^\times$, i.e.,
$ \PrPol \ltimes_\theta \PolFun^\times$ contains an isomorphic copy of 
$\mathcal{P}_R(R[\alpha])$. To do so let us first prove the following lemma,
which is a   special case of \cite[Theorem 2.8]{haki}.

\begin{lemma}\label{6}
	Let   $g\in R[x]$. Then $g$ permutes $R[\alpha]$ if and only if
	$g$ permutes $R$ and $g'$ is a unit-valued polynomial.
\end{lemma}

\begin{proof}
	$(\Rightarrow)$
	Let $c\in R$. Then $c\in R[\alpha]$. Since $g$
	permutes  $R[\alpha] $, there exist       
	$a,b \in R$ such that $g(a+b \alpha)= c$.
	Thus $g(a)+ bg'(a)\alpha = c$ by Fact~\ref{3}.
	So $g(a)=c$, and therefore $g$ is onto on the ring $R$, and hence a permutation polynomial on $R$.
	
	Suppose that $g$ is  not a unit-valued polynomial. Then there exists  $a\in R$ such that 
	$g'(a)$ is a zero divisor of $R$.
	For, if $0\ne b\in R$ such that $bg'(a)=0$,
	then by Fact~\ref{3},  $$g(a+b\alpha)=g(a)+bg'(a)\alpha =g(a).$$
	So $g$ does not permute $R[\alpha]$, which is a contradiction.\\
	($\Leftarrow$) It is enough to show that $g$ is injective.
	For, if $a,b,c,d \in R$ such that
	$g(a+b\alpha)=g(c+d\alpha)$, then by Fact~\ref{3},
	$$g(a)+bg'(a)\alpha = g(c)+dg'(c)\alpha.$$
	Then  we have $g(a)= g(c)$ and $bg'(a)= dg'(c)$.
	Hence $a= c$   since $g$  permutes  $R$.
	Then, since $g'(a)$ is a unit of $R$,  $b=d$ follows.               
\end{proof}

\begin{remark}\label{7}
	Let $F\in \mathcal{P}_R(R[\alpha])$. Then there exists $f\in R[x]$ such that $F=[f]_{R[\alpha]}$
	by definition~\ref{5}. So for	every $a,b\in R$ we have, by Fact~\ref{3},
	$$F(a+b\alpha)=f(a+b\alpha)=f(a)+bf'(a)\alpha= [f]_R(a)+b[f']_R(a),$$ 
	where $[f]_R$ and $[f']_R$ denote  the polynomial functions induced  by $f$ and $f'$ on $R$ 
	respectively. This shows that the pair $([f]_R,[f']_R)$ determines	$F=[f]_{R[\alpha]}$ completely. 
	Therefor, if $g\in R[x]$ is another polynomial representing $F$, 
	then $([g]_R,[g']_R)=([f]_R,[f']_R)$.
\end{remark}

\begin{lemma}\label{8}
	Let $f,g\in R[x]$. Then $f\circ g(a+b\alpha)=f(g(a)) +bg'(a)f'(g(a))\alpha$ 
	for every $a,b\in R$. Moreover, if $[f\circ g]_{R[\alpha]}$ the polynomial function
	represented by $f\circ g$ over $R[\alpha]$, then $[f\circ g]_{R[\alpha]}$ is completely
	determined by the pair $([f\circ g]_R,[f'\circ g]_R\cdot [g']_R)$.
\end{lemma}

\begin{proof}
	The first statement follows by applying Fact~\ref{3} and the chain rule to the polynomial $f\circ g$.
	\begin{comment}
	Let $a,b\in R$. Then by applying Fact~\ref{3},  $g(a+b\alpha)=g(a)+bg'(a)\alpha$. 
	Set $c=g(a) $, $d=bg'(a)$. Again by Fact~\ref{3},
	$$f(c+d\alpha)=f(c)+df'(c)\alpha=f(g(a))+bg'(a)f'(g(a))\alpha.$$
	\end{comment}
	The last statement follows from the first part and Remark~\ref{7}.      
\end{proof}

Recall from Definition~\ref{5} and Fact ~\ref{2}  the definitions of the groups
$\mathcal{P}_R(R[\alpha])$ and $\UPF$ respectively.
\begin{proposition}\label{9}
	Let $R$ be a finite  commutative ring. The group of  polynomial permutations 
	$\mathcal{P}_R(R[\alpha])$ is embedded in  $ \PrPol \ltimes_\theta \PolFun^\times$, 
	where $\theta$ is the homomorphism defined in Lemma~\ref{7-aut}.
\end{proposition}

\begin{proof}
	Let $F\in \mathcal{P}_R(R[\alpha])$. Then by Definition~\ref{5}, $F$ is induced
	by a polynomial $f\in R[x]$. Hence $f$ permutes $R$ and $f'$ is a unit-valued 
	polynomial by Lemma~\ref{6}. Thus $[f]_R\in \PrPol[R]$ and $[f']_R\in \UPF$. Now  
	define a map
	\[\phi\colon\mathcal{P}_R(R[\alpha]) \longrightarrow   \PrPol \ltimes_\theta \PolFun^\times \text{ by }
	\phi(F)=([f]_R,[f']_R). \]
	By Remark~\ref{7}, $\phi$ is well defined. 
	Thus, if $F_1\in \mathcal{P}_R(R[\alpha])$ is induced by $f_1\in R[x]$, then we have
	$F\circ F_1$ is determined by the pair
	$([f\circ f_1  ]_R, [f'\circ f_1]_R\cdot [f_1']_R)$ by Lemma~\ref{8}. 
	Therefore, by the operation of $\PrPol \ltimes_\theta \PolFun^\times$ defined in 
	Proposition~\ref{semiconst}, 
	\begin{align*}
	\phi[F\circ F_1]&=( [f\circ f_1]_R,[f'\circ f_1]_R\cdot 
	[f_1']_R)=\big([f]_R\circ [f_1]_R,([f']_R\circ [f_1]_R)\cdot [f_1']_R\big)\\
	&=( [f]_R,[f']_R)([f_1]_R,[f'_1]_R )=\phi(F)\phi(F_1).
	\end{align*}
	Thus   $\phi$ is a homomorphism.
	Now  suppose that $F\ne F_1$. Then by  Remark~\ref{7},
	$( [f]_R, [f']_R)\ne([f_1]_R,[f'_1]_R)$. Hence $\phi$ is injective.
\end{proof}

%When $R$ is a finite field we show later that  $\mathcal{P}_R(R[\alpha])\cong \UPF \rtimes_\theta\PrPol $

\section{ The pointwise stabilizer group of $R$ and the group $\UPF$}\label{finitefieldcase}

In this section, we show that the group of unit-valued polynomial functions contains
an isomorphic copy of the pointwise stabilizer  group of $R$ (defined below). 
In particular, when $R=\mathbb{F}_q$ the finite field of $q$ elements, we prove
that $\UPF[\mathbb{F}_q]$ is isomorphic to this group. We employ this result in
the end of this section to prove	that 
$\mathcal{P}_{\mathbb{F}_q}(\mathbb{F}_q[\alpha])\cong 
\PrPol[\mathbb{F}_q]\ltimes_\theta \PolFun[\mathbb{F}_q]^\times$. 
%Furthermore, we show that $\UPF$ contains an isomorphic copy of the pointwise stabilizer group of $R$ defined below.

Now we recall the definition of the pointwise stabilizer group of $R$ from \cite{haki}.

\begin{definition}\label{3-1} %\cite[Defintion 2.13]{haki}
	Let
	$\Stab=\{F\in \mathcal{P}(R[\alpha]):
	F(r)=r \text{ for every } r\in R\}$.
\end{definition}

It is evident that $\Stab[R]$ is closed under composition, and hence it  is a subgroup
of $\mathcal{P}(R[\alpha])$ since	it is a non-empty finite set. We call this group 
the pointwise stabilizer group of $R$. 

Recall from the introduction that the ideal $N_R$ is   consisting of all null polynomials on $R$.

\begin{lemma}\label{3-2}
	Let $g,h\in R[x]$. Suppose that  $[g]_R=[h]_R$.
	Then there exists   $f \in N_R$ such that $g=h+f$.
\end{lemma}

\begin{proof}
	Let $f=g-h$. Then $[f]_R$ is the constant zero function on $R$, that is $f \in N_R$. 
\end{proof}

We need the following proposition from~\cite{haki}.  However, we prove it  as
the proof does not depend on extra materials.

\begin{proposition}\cite[Proposition 2.15]{haki}\label{3-3}
	Let $R$ be a finite commutative  ring. Then \[\Stab=\{F\in \mathcal{P}(R[\alpha]):F
	\textnormal{ is induced by } x+g(x), g \in N_R \}.\]
	Moreover, $\Stab$ is a subgroup of $\mathcal{P}_R(R[\alpha])$.
\end{proposition}

\begin{proof}
	Obviously, 
	\[\Stab[R]\supseteq\{F\in \mathcal{P}(R[\alpha]):F
	\textnormal{ is induced by } x+g(x), g \in N_R \}.\]
	Now  if $F\in \Stab$, then by Definition~\ref{3-1}, $F\in \mathcal{P}(R[\alpha])$
	such that $F(r)=r$ for each $r\in R$; and
	$F$ is induced by a polynomial $h\in R[\alfa][x]$. By Fact~\ref{3}-(\ref{3-b}), $h=h_0+ h_1\alfa$; 
	where $h_0,h_1\in R[x]$, and so
	$r=F(r)=h_0(r)+h_1(r)\alfa$ for every $r\in R$.
	It follows that $h_1(r)=0$ for every $r\in R$, i.e.,
	$h_1$ is  null  on $R$. Hence  $h_1\alpha$ is  null  on $R[\alpha]$ by Fact~\ref{4}. 
	Thus  $[h_0]_{R[\alfa]}=[h_0+h_1\alpha]_{R[\alfa]}=F$, 
	that is, $F$ is induced by $h_0$.
	Also, $h_0\equiv x \mod N_R$, that is $[h_0]_R=id_R$,  and  therefore 
	$h_0(x)=x+f(x)$ for some $f\in N_R$  by Lemma~\ref{3-2}. This shows the other inclusion.
	
	The last statement follows from the previous part and the fact that
	$\Stab[R]$ and $\mathcal{P}_R(R[\alpha])$ are subgroups of $\mathcal{P}(R[\alpha])$.
\end{proof}

\begin{remark}\label{Lagpol}
	Let $\mathbb{F}_q=\{a_0,\ldots,a_{q-1}\}$ be the finite field with $q$ elements. 
	If $F\colon\mathbb{F}_q\longrightarrow \mathbb{F}_q$,  then the polynomial 
	$f(x)=\sum_{i=0}^{q-1} F(a_i)\prod_{\substack{j=0\\ j\ne i}}^{q-1}\frac{x-a_j}{a_i-a_j}\in \mathbb{F}_q[x]$
	represents $F$. Such a polynomial is called the Lagrange's polynomial and this method
	of construction is called Lagrange's interpolation. Therefore every function on
	a finite field is a polynomial function, and hence $|\PolFun[\mathbb{F}_q]|=q^q$. 
	In particular, every permutation (bijection) on $\mathbb{F}_q$ is a polynomial
	permutation, and so $|\PrPol[\mathbb{F}_q]|= q!$. Further, every unit-valued function
	is a unit-valued polynomial function, and thus  $|\UPF[\mathbb{F}_q]|= (q-1)^q$ since 
	$ \mathbb{F}_q^\times =\mathbb{F}_q\setminus\{0\}$. Moreover, it is obvious that
	Lagrange's interpolation assigns to every function on $\mathbb{F}_q$ a unique 
	polynomial of degree at most $q-1$. Hence every  polynomial of degree at most $q-1$
	is  the Lagrange's polynomial of a function on $\mathbb{F}_q$ since the number of
	these polynomials is $q^q$, which is the number of functions on $\mathbb{F}_q$.   
\end{remark}

\begin{lemma}\label{3-5}
	%        Let $\mathbb{F}_q$ be the finite field with $q$ elements.       Then
	For each pair of functions $(G,F)$ with
	\[G\colon\mathbb{F}_q\longrightarrow \mathbb{F}_q \text{ bijective  and } 
	F\colon\mathbb{F}_q\longrightarrow \mathbb{F}_q\setminus\{0\}\]
	there exists a polynomial  $g\in \mathbb{F}_q[x]$ such that
	$(G,F)=([g]_{\mathbb{F}_q},[g']_{\mathbb{F}_q})$. Furthermore, $g$ is a permutation polynomial on
	$\mathbb{F}_q[\alfa]$, i.e., $[g]_{\mathbb{F}_q[\alfa]}\in  \mathcal{P}_{\mathbb{F}_q}(\mathbb{F}_q[\alpha])$.
\end{lemma}

\begin{proof}
	By Lagrange's interpolation (Remark~\ref{Lagpol}), there exist  two polynomials $f_0,f_1\in \mathbb{F}_q[x]$   
	such that $[f_0]_{\mathbb{F}_q} =G$
	and $[f_1]_{\mathbb{F}_q} =F$. Then set
	\[g(x) = f_0(x) + (f'_0(x) - f_1(x))(x^q-x).\]
	Thus  \[g'(x) =(f''_0(x) - f'_1(x))(x^q-x)+f_1(x),\]
	whence $[g]_{\mathbb{F}_q}=[f_0]_{\mathbb{F}_q}=G$ and 
	$[g']_{\mathbb{F}_q}=[f_1]_{\mathbb{F}_q}=F$ since $(x^q-x)$
	is a null polynomial on $\mathbb{F}_q$.
	The last part follows by Lemma~\ref{6} and Definition~\ref{5}.
\end{proof}

\begin{remark}\label{3-4}
	Let $F,F_1\in \Stab$, with $F\ne F_1$. Then, by Proposition~\ref{3-3},
	$F$ and $F_1$ are induced by $x+f(x)$ and $x+g(x)$ respectively for some $f,g\in N_R$.
	By Definition~\ref{3-1}, $[x+f(x)]_R=[x+g(x)]_R= id_R$. 
	Thus  by Remark~\ref{7}, since $F\ne F_1$, $[1+f']_R\ne[1+g']_R$.
\end{remark}

Throughout for any set $A$ let $|A|$ denote the number of elements in  $A$.

\begin{theorem}\label{3-6}
	The  pointwise stabilizer group of $R$ is embedded in the group of
	unit-valued polynomial functions $\UPF$. In particular, if  $R=\mathbb{F}_q$ is  the finite
	field of $q$ elements, then $\Stab[{\mathbb{F}_q}] \cong \UPF[{\mathbb{F}_q}]$.
\end{theorem}

\begin{proof}
	Let $F\in \Stab$. Then $F$ is induced by $x+f(x)$ for some $f\in N_R$ by Proposition~\ref{3-3}.
	By Lemma~\ref{6}, $1+f'(x) $ is a unit-valued polynomial, whence $[1+f']_R\in \UPF$. 
	Define a map $\phi\colon \Stab \longrightarrow \UPF$ by $\phi(F)=[1+f']_R$.  
	By Remark~\ref{7}, $\phi$ is well defined;  and it is injective by Remark~\ref{3-4}.
	Now if $F_1\in \Stab$, then $F_1$ is induced by $x+g(x)$ for some $g\in N_R$. 
	By routine calculations and   Fact~\ref{3}, one shows easily that
	\[F\circ F_1(a+b\alpha)=a+b\big(1+f'(a)+g'(a)+f'(a)g'(a)\big)\alpha  \text{ for each } a,b\in R.\]
	Hence $F\circ F_1$ is induced by $x+f(x) +g(x)+f'(x)g(x)$. Since $f''g\in N_R$, $[f''g]_R$ 
	is the zero function on $R$, and whence
	\begin{align*}
	\phi(F\circ F_1) &=[ 1+f'+ g' +f'g'+f''g ]_R=[1+f'+ g' +f'g']_R+[f''g]_R\\
	&=[(1+f')(1+g')]_R
	=[1+f']_R\cdot[1+g']_R=\phi(F)\cdot\phi(F_1).
	\end{align*}
	Thus  $\phi$ is a homomorphism. Therefore $\Stab$ is embedded in $\UPF$.
	
	For the case $R={\mathbb{F}_q}$, we need only to show that $\phi$ is surjective. 
	Let $F\in \UPF[{\mathbb{F}_q}]$. Then there exists $f\in \mathbb{F}_q[x]$ such that 
	$[f]_ {\mathbb{F}_q}=id_{\mathbb{F}_q}$, $[f']_{\mathbb{F}_q}=F$ and 
	$[f]_{{\mathbb{F}_q}[\alpha]}\in \mathcal{P}_{\mathbb{F}_q}(\mathbb{F}_q[\alpha])$
	by Lemma~\ref{3-5}.  Thus $[f]_{{\mathbb{F}_q}[\alpha]}\in \Stab[\mathbb{F}_q]$ 
	by Definition~\ref{3-1}, and hence
	$\phi([f]_{{\mathbb{F}_q}[\alpha]})=[f']_{\mathbb{F}_q}=F$. Therefore $\phi$ is surjective.
\end{proof}

\begin{corollary}\label{3-61}
	$\Stab \cong \UPF$ if and only if  $\PrPol \ltimes_\theta \Stab\cong  \PrPol \ltimes_\theta \UPF$.
\end{corollary}

\begin{proof}
	$(\Rightarrow)$ Obvious.
	
	$(\Leftarrow)$ Assume that         $\PrPol \ltimes_\theta \Stab\cong  \PrPol \ltimes_\theta \UPF$. 
	Then $|\Stab|= |\UPF|$. But, $\UPF$ contains a subgroup isomorphic to $\Stab$ 
	by Theorem~\ref{3-6}, and hence  $\Stab \cong \UPF$.
\end{proof}

%Keep the notations of Definitions~\ref{5} and ~\ref{3-1} with $R$ is replaced by $\mathbb{F}_q$.
In Proposition~\ref{9} we have proved for any finite ring $R$ that 
the group $\mathcal{P}_R(R[\alpha])$ is embedded in 
% a semi direct product of the group $\UPF$ by $\PrPol[R]$,
$\PrPol \ltimes_\theta \UPF$.  In the following theorem we show that, 
for a finite field $\mathbb{F}_q$,
\[\mathcal{P}_{\mathbb{F}_q}(\mathbb{F}_q[\alpha])\cong
\PrPol[\mathbb{F}_q] \ltimes_\theta \PolFun[\mathbb{F}_q]^\times.\]

\begin{theorem}\label{3-7}
	Let $\mathbb{F}_q$ be the finite field of $q$ elements, and let  $\theta$ be
	the homomorphism defined in Lemma~\ref{7-aut}. Then 
	$$\mathcal{P}_{\mathbb{F}_q}(\mathbb{F}_q[\alpha])\cong \PrPol[\mathbb{F}_q] \ltimes_\theta \PolFun[\mathbb{F}_q]^\times \cong \PrPol[\mathbb{F}_q]\ltimes_\theta \Stab[\mathbb{F}_q].$$
\end{theorem}

\begin{proof}
	In view of Proposition~\ref{9} and Theorem~\ref{3-6} we need only to show that
	\[|\mathcal{P}_{\mathbb{F}_q}(\mathbb{F}_q[\alpha])|\ge| \PolFun[\mathbb{F}_q]^\times ||\PrPol[\mathbb{F}_q]|.\]
	
	\begin{comment}
	Since every function over finite field is a polynomial function one proves easily that $|\PrPol[\mathbb{F}_q]|=q!$ and $|\PolFun[\mathbb{F}_q]^\times|=(q-1)^q$.
	\end{comment}
	
	Hence, by Remark~\ref{Lagpol},  it is sufficient to show that 
	$|\mathcal{P}_{\mathbb{F}_q}(\mathbb{F}_q[\alpha])|\ge q!(q-1)^q$.
	
	Now consider the pair of functions $(G,F)$ with
	\[G\colon\mathbb{F}_q\longrightarrow \mathbb{F}_q \text{ bijective  and }
	F\colon\mathbb{F}_q\longrightarrow \mathbb{F}_q\setminus\{0\}.\]  
	It is obvious that the total number of different pairs of these form is $q!(q-1)^q$. 
	Moreover, by Lemma~\ref{3-5},  there exists $g\in  \mathbb{F}_q[x]$  such that
	$(G,F)=([g]_{\mathbb{F}_q},[g']_{\mathbb{F}_q})$ and 
	$[g]_{\mathbb{F}_q[\alfa]}\in  \mathcal{P}_{\mathbb{F}_q}(\mathbb{F}_q[\alpha])$. 
	By Remark~\ref{7}, every two different pairs of functions of the form of Lemma~\ref{3-5}
	determine two different elements of   $\mathcal{P}_{\mathbb{F}_q}(\mathbb{F}_q[\alpha])$. 
	Therefore $|\mathcal{P}_{\mathbb{F}_q}(\mathbb{F}_q[\alpha])|\ge q!(q-1)^q $.
\end{proof}
We conclude this section with the following remark.
\begin{remark}
	It should be mentioned that, when $q=p$, where $p$ is a prime number, Frisch and Kren~\cite{per2} used
	the Sylow $p$-groups
	of $\mathcal{P}_{\mathbb{F}_p}(\mathbb{F}_p[\alpha])\cong \PrPol[\mathbb{F}_p] \ltimes_\theta \PolFun[\mathbb{F}_p]^\times $ to determine the number of the Sylow $p$-groups
	of $\PrPol[\mathbb{Z}_{p^n}]$ %, where $\mathbb{Z}_{p^n}$ denotes the ring of integers modulo $p^n$, 
	for $n\ge 2$.
\end{remark}
\begin{comment}
\section{A counter example}
In this section, we provide a counter example to show that  in general $\Stab$ and $\UPF$ are not isomorphic. Then we show in this case that the embedding of Proposition~\ref{3-3} is not surjective.
\end{comment}

\section{The number of unit-valued polynomial functions on the ring $\mathbb{Z}_{p^n}$}\label{the num}

Throughout this section let $p$ be a prime number and $n$ be a positive integer.
In the literature authors considered the number of polynomial functions and 
polynomial permutations on  the ring of integers modulo $p^n$, however they overlooked
counting  the number of unit-valued polynomial functions modulo $p^n$
(see for example,~\cite{pol1,pol2}). In this section we apply the results of~\cite{pol1}
to derive an explicit formula for the order of the group $\UPF[{\mathbb{Z}_{p^n}}]$,
i.e., the number of unit-valued polynomial functions $\pmod {p^n}$. In addition to that,
we find canonical representations for these functions.

It is well known that the ring of integers $\mathbb{Z}$ can be viewed as a set
of representatives of the elements of $\mathbb{Z}_{p^n}$. This allows us to represent
the polynomial functions on $\mathbb{Z}_{p^n}$ by polynomials from $\mathbb{Z}[x]$. 
To simplify our notation we use the symbol $[f]_{p^n}$ instead of
$[f]_{\mathbb{Z}_{p^n}}$ to indicate the function induced by
$f\in \mathbb{Z}[x]$ on $\mathbb{Z}_{p^n}$.

\begin{lemma}\label{6-1}
	Let $n\ge 2$ and let $f\in \mathbb{Z}[x]$. Then $[f]_{p^n}\in \UPF[{\mathbb{Z}_{p^n}}]$ 
	if and only if $[f]_{p^{n-1}}\in \UPF[{\mathbb{Z}_{p^{n-1}}}]$ if and only if
	$[f]_{p}\in \UPF[{\mathbb{Z}_{p}}]$.
\end{lemma}

\begin{proof}
	Let $f\in \mathbb{Z}[x]$ and let $k\ge 1$. Then $[f]_{p^k}\in \UPF[{\mathbb{Z}_{p^k}}]$
	if and only if $f(r)\in \mathbb{Z}_{p^k}^{\times}$ for every $r\in \mathbb{Z}_{p^k}$
	(by Definition~\ref{1}) if and only if $\gcd(p^k, f(r))=1$ 
	for every $r\in \mathbb{Z}$ if and only if $f(r)\not\equiv 0 \pmod{p}$ 
	for each $r\in \mathbb{Z}$ if and only if $[f]_p\in\UPF[\mathbb{Z}_p]$.
\end{proof}

\begin{remark}\label{6-2}
	Let $n>1$. Define a map 
	$\phi_n\colon\PolFun[\mathbb{Z}_{p^n}]\longrightarrow \PolFun[\mathbb{Z}_{p^{n-1}}]$ 
	by  $\phi_n(F)=[f]_{p^{n-1}}$, where $f\in \mathbb{Z}[x]$  such that $F=[f]_{p^n}$.
	Now  if  $f,g\in \mathbb{Z}[x]$ such that $[g]_{p^n}=[f]_{p^n}$, then
	$[g-f]_{p^n}$ is the zero function; that is, $g-f$ is a null polynomial on $\mathbb{Z}_{p^n}$.
	Thus $g-f$ is a null polynomial on $\mathbb{Z}_{p^{n-1}}$, whence  $[g]_{p^{n-1}}=[f]_{p^{n-1}}$.
	This shows that the map $\phi_n$  is well defined. Evidently,  $\phi_n$ is an additive group
	epimorphism  with	$|\PolFun[\mathbb{Z}_{p^n}]|=|\PolFun[\mathbb{Z}_{p^{n-1}}]||\ker\phi_n|$. 
	Moreover, if $n>2$ and $f\in \mathbb{Z}[x]$ such that $[f]_{p^n}\in \ker \phi_n$, 
	then it is  obvious that $[f]_{p^{n-1}}\in\ker \phi_{n-1}$; hence
	$[f]_{p^k}\in \ker \phi_k$ for $2\le k<n$. But not vice versa, for example, 
	if $f(x)=p^3+p^4x$, $[f]_{p^4}\in \ker \phi_4$ since $\phi_4([f]_{p^4})=[f]_{p^3}=[0]_{p^3}$
	is the zero function on $\mathbb{Z}_{p^3}$. However, 
	$\phi_5([f]_{p^5})=[f]_{p^4}=[p^3]_{p^4}$; i.e, the constant function which sends every
	element of $\mathbb{Z}_{p^4}$ to $p^3\ne 0$, that is $[f]_{p^5}\notin \ker \phi_5$.
\end{remark}

\begin{notation}\label{nota}
	In the rest of the paper  let  $\beta(n)$ denote the smallest positive integer $k$
	such that $p^n\mid k!$, while $v_p(n)$ denotes the largest integer $s$ such that $p^s\mid n$.
	
	Let $(x)_0=1$, and let  $(x)_j=x(x-1)(x-2)\cdots(x-j+1)$ for any positive integer $j$.       
\end{notation}

The following lemma from \cite{pol1} gives the cardinality of $\ker \phi_n$ of 
the epimorphism $\phi_n$ mentioned in Remark~\ref{6-2}.

\begin{lemma}\cite[Theorem 2]{pol1}\label{6-3}
	Let $n>1$. Every element in $\ker\phi_n$ can be  represented  by a unique polynomial of 
	the form $\sum_{i+v_p(j!)=n-1}a_{ij} p^i(x)_j$ where $i,j\ge 0$ and $0\le a_{ij}\le p-1$.
	Moreover, $|\ker \phi_n|=p^{\beta(n)}$.
\end{lemma}

Similar to \cite[Corollary 2.1]{pol1} we have for unit-valued polynomial functions modulo ${p^n}$.

\begin{lemma}\label{6-4}
	Let $n>1$. Then $|\UPF[{\mathbb{Z}_{p^n}}]|=p^{\beta(n)}|\UPF[{\mathbb{Z}_{p^{n-1}}}]|$.
\end{lemma}

\begin{proof}
	Consider the epimorphism $\phi_n$ given in Remark~\ref{6-2}. We have, by Lemma~\ref{6-1},\\
	$\phi_n^{-1}(\UPF[{\mathbb{Z}_{p^{n-1}}}])=\UPF[{\mathbb{Z}_{p^n}}]$. Hence
	$|\UPF[{\mathbb{Z}_{p^n}}]|=|\phi_n^{-1}(\UPF[{\mathbb{Z}_{p^{n-1}}}])|$. Now  if
	$F\in \UPF[{\mathbb{Z}_{p^{n-1}}}]$ then,  by Lemma~\ref{6-3},
	$|\phi_n^{-1}(F)|=|\ker \phi_n|$. Therefore
	$$|\UPF[{\mathbb{Z}_{p^n}}]|=|\phi_n^{-1}(\UPF[{\mathbb{Z}_{p^{n-1}}}])|=
	|\ker \phi_n||\UPF[{\mathbb{Z}_{p^{n-1}}}]|.$$ 
	The result now follows from Lemma~\ref{6-3}.
\end{proof}

Keep the notations of Fact~\ref{2} and  Notation~\ref{nota}. We now state our 
counting formula for the order of $\UPF[{\mathbb{Z}_{p^{n-1}}}]$.

\begin{theorem}\label{6-5}
	Let $n> 1$ and let $\UPF[{\mathbb{Z}_{p^n}}]$ be the group of unit-valued polynomial functions modulo $p^n$. 
	Then  $|\UPF[{\mathbb{Z}_{p^n}}]|=(p-1)^p{p^{\sum_{k=2}^{n}\beta (k)}}$.
\end{theorem}

\begin{proof}
	By applying Lemma~\ref{6-4}, $n-1$ times, we have that
	$|\UPF[{\mathbb{Z}_{p^n}}]|=|\UPF[{\mathbb{Z}_{p}}]|{p^{\sum_{k=2}^{n}\beta (k)}}$.
	
	But  $|\UPF[{\mathbb{Z}_{p}}]|=(p-1)^p$  by Remark~\ref{Lagpol}.
\end{proof}

\begin{lemma}\label{6-8}
	Let $1\le n\le k$ and let $f,g\in \mathbb{Z}[x]$. If $[f]_{p^n}\ne [g]_{p^n}$
	then $[f]_{p^k}\ne [g]_{p^k}$.
\end{lemma}

\begin{proof}
	Assume to the contrary that $[f]_{p^k}=[g]_{p^k}$. Then $f(a)\equiv g(a)\pmod{p^k}$
	for every $a\in\mathbb{Z}$, and hence $f(a)\equiv g(a)\pmod{p^n}$ for every 
	$a\in\mathbb{Z}$ since $n\le k$. But this means that $[f]_{p^n}= [g]_{p^n}$ which is a contradiction.
\end{proof}

\begin{remark}\label{6-7}
	\begin{comment}
	It is  a well known fact that every polynomial function over the finite field $\mathbb{F}_q$ of $q$ elements is represented uniquely by a polynomial of the form $\sum_{i=0}^{q-1}a_ix^i$, where $a_i\in  \mathbb{F}_q$.  In particular,
	\end{comment}
	
	Recall from Remark~\ref{Lagpol} that every polynomial function on the 
	finite filed $ \mathbb{F}_p=\mathbb{Z}_p$ is represented uniquely by a polynomial of the form
	$\sum_{i=0}^{p-1}a_ix^i$, where $0\le a_i\le  p-1$. Such a representation is obtained 
	by Lagrange's interpolation. Now let $l_1,\ldots,l_{(p-1)^p}$ be the unique polynomials
	representations of the elements of $\UPF[\mathbb{Z}_p]$ obtained by Lagrange's interpolation.
	Then obviously, $[l_i]_{p^n}\ne [l_j]_{p^n}$ if and only if $i\ne j$ for every $n\ge 1$
	by Lemma~\ref{6-8}.
	% Furthermore, if   $h,g\in \ker\phi_{n+1}$, then  $[l_i+h]_{p^n}\ne [l_j+g]_{p^n}$ if and only if $i\ne j$ by the definition of the homomorphism $\phi_{n+1}$.
\end{remark}

We need the following fact from~\cite{pol1} in which we use our notations.

\begin{lemma} \cite[Theorem 1]{pol1}\label{simplifyproof}
	If $F\in \PolFun[\mathbb{Z}_{p^n}]$, there exists one and only one polynomial
	$f\in\mathbb{Z}[x]$ with $[f]_{p^n}=F$ and $f=\sum_{i+v_p(j!)<n}a_{ij} p^i(x)_j$
	such that $i,j\ge 0$ and $0\le a_{ij}\le p-1$.
\end{lemma}

Keep the notations of Remark~\ref{6-2}, Notation~\ref{nota}, and Remark~\ref{6-7}.
The following theorem gives canonical representations for the elements of
$\UPF[\mathbb{Z}_{p^n}]$ as linear combinations of the unique representations of 
the elements of $\UPF[\mathbb{Z}_{p}]$ and those of the elements of $\ker \phi_k$ 
for $k=2,\ldots, n$, mentioned in Remark~\ref{6-7} and  Lemma~\ref{6-3}, respectively.

\begin{theorem}
	Let $n\ge 2$. Then every element in $\UPF[\mathbb{Z}_{p^n}]$ can be uniquely represented 
	by a polynomial of the form
	\begin{equation}\label{eq}
	l_s(x)+\sum_{k=2}^{n}\ \sum_{i+v_p(j!)=k-1}a_{kij} p^i(x)_j, 
	\text{ where } i,j\ge 0;\  0\le a_{kij}\le p-1 \text{ and }s=1,\ldots,(p-1)^p.
	\end{equation}
	
\end{theorem}

\begin{proof}
	We have  $$[l_s(x)+\sum_{k=2}^{n}\ \sum_{i+v_p(j!)=k-1}a_{kij} p^i(x)_j]_p=
	[l_s(x)]_p+[\sum_{k=2}^{n}\ \sum_{i+v_p(j!)=k-1}a_{kij} p^i(x)_j]_p=[l_s]_p\in \UPF[\mathbb{Z}_p].$$
	Thus every polynomial of the form (\ref{eq}) represents an element of $\UPF[\mathbb{Z}_{p^n}]$
	by Lemma~\ref{6-1}.
	
	Fix an integer $2\le k\le n$. Then, by Lemma~\ref{6-3}, the sum $\sum_{i+v_p(j!)=k-1}a_{kij} p^i(x)_j$
	is a member of $\ker\phi_k$, and hence can be chosen in $|\ker\phi_k|$ ways.
	Therefore the total number of polynomials of the form (\ref{eq}) is
	$(p-1)^p\prod_{k=2}^{n}|\ker\phi_k|=|\UPF[\mathbb{Z}_{p^n}]|$  by Theorem~\ref{6-5}.
	
	So to complete the proof
	we need only to show that every non equal two polynomials of the form (\ref{eq})
	induce two different functions on $\mathbb{Z}_{p^n}$. Now let $f,g\in \mathbb{Z}[x]$ of
	the form (\ref{eq}) with $f\ne g$. Then,
	for simplicity, we write $$f(x)=l_r(x)+\sum_{k=2}^{n} f_k(x) \text{ and }g(x) =l_s(x)+\sum_{k=2}^{n} g_k(x),
	\text{ where }[f_k]_{p^k},[g_k]_{p^k}\in \ker \phi_k,\ k=2,\ldots,n.$$ First, we notice that if $r\ne s$,
	then  by Remark~\ref{6-7}, $[f]_{p}=[l_r]_p\ne [l_s]_p= [g]_{p}$. Thus $[f]_{p^n}\ne  [g]_{p^n}$
	by Lemma~\ref{6-8}. So we may assume that $r=s$, and hence 
	$\sum_{k=2}^{n} f_k(x) \ne \sum_{k=2}^{n} g_k(x)$. Since
	$\sum_{k=2}^{n} f_k(x)\ne\sum_{k=2}^{n} g_k(x)$, there exists 
	$2\le k_0\le n$ such that $f_{k_0} \ne g_{k_0}$	and we choose $k_0$ to be minimal
	with this property. Now  $f_{k_0}$, $g_{k_0}$ are two different polynomials of the form 
	\begin{equation}\label{eqf}
	\sum_{i+v_p(j!)=k_0-1}a_{k_0ij} p^i(x)_j,  \text{ where } i,j\ge 0;\  0\le a_{k_0ij}\le p-1.
	\end{equation} 
	Then by Lemma~\ref{simplifyproof}, $[f_{k_0}]_{p^{k_0}}\ne [g_{k_0}]_{p^{k_0}}$; 
	and  this  together with the fact   $f_k=g_k$, for $k<k_0$, implies that
	$[l_s+\sum_{k=2}^{k_0} f_k ]_{p^{k_0}}\ne [l_s+\sum_{k=2}^{k_0} g_k ]_{p^{k_0}}$.
	Hence by the definition of the map $\phi_{k_0+1}$,
	$$[f]_{p^{k_0}}=\phi_{k_0+1}([f]_{p^{k_0+1}})
	=[l_s+\sum_{k=2}^{k_0} f_k ]_{p^{k_0}}\ne [l_s+\sum_{k=2}^{k_0} g_k
	]_{p^{k_0}}=\phi_{k_0+1}([g]_{p^{k_0+1}})=[g]_{p^{k_0}},$$ 
	that is, $[f]_{p^{k_0}}\ne [g]_{p^{k_0}}$.  Therefore  $[f]_{p^n}\ne  [g]_{p^n}$ 
	by Lemma~\ref{6-8}, this completes the proof.	      
\end{proof}

\begin{remark}
	Let $R=\mathbb{Z}_{4}=\{0,1,2,3\}$. In this case, 
	$\mathbb{Z}_{4}[\alpha]=\{a+b\alpha:a,b \in \mathbb{Z}_{4}\}$. 
	Consider now the polynomial $f(x)=(x^2-x)^2$. By  Fermat's little theorem,
	$f$ is a null polynomial on $\mathbb{Z}_{4}$; hence every unit-valued polynomial
	function is induced by a polynomial of degree less than $4$. Next we show that $f$ 
	is  null  on  $\mathbb{Z}_{4}[\alpha]$. So, if $a,b\in \mathbb{Z}_{4}$, then
	\begin{align*}
	f(a+b\alpha) &=\big((a+b\alpha)^2-(a+b\alpha)\big)^2=\big((a^2+2ab\alpha)-(a+b\alpha)\big)^2\\
	&=\big((a^2-a)+(2ab-b)\alpha\big)^2=(a^2-a)^2+ 2(a^2-a)(2ab-b)\alpha=0. 
	\end{align*}
	Thus $f$ is  null on $\mathbb{Z}_{4}[\alpha]$, whence every polynomial function
	on $\mathbb{Z}_{4}[\alpha]$ is represented by a polynomial of degree less than $4$.
	The null polynomials on $\mathbb{Z}_{4}$ of degree less than $4$ are
	\[f_1=0,\ f_2 =2(x^2-x),\ f_3=2(x^3-x) \text{ and } f_4= 2(x^3-x^2).\] Then simple 
	calculations shows that  $1+f_1',\ldots, 1+f_4'$ induce four different 
	unit-valued functions on $\mathbb{Z}_{4} $. Thus $|\Stab[\mathbb{Z}_{4}]|=4$, but
	$|\UPF[\mathbb{Z}_{4}]|=2^{\beta(2)}=16$ by Theorem~\ref{6-7}. 
	Furthermore, a routine verification (or see \cite[Lemma 5.7]{haki} for
	a more general case) shows that there is an epimorphism from 
	$\mathcal{P}_{\mathbb{Z}_4}(\mathbb{Z}_{4}[\alpha])$ onto $\PrPol[\mathbb{Z}_{4}]$ 
	which admits $\Stab[\mathbb{Z}_4]$ as a kernel. Thus 
	$|\mathcal{P}_{\mathbb{Z}_4}(\mathbb{Z}_{4}[\alpha])|=
	|\Stab[\mathbb{Z}_{4}]||\PrPol[\mathbb{Z}_{4}]|$, and hence
	\[| \PrPol[\mathbb{Z}_{4}] \ltimes_\theta \UPF[\mathbb{Z}_{4}]|=
	|\PrPol[\mathbb{Z}_{4}]||\UPF[\mathbb{Z}_{4}] |>|\Stab[\mathbb{Z}_{4}]||\PrPol[\mathbb{Z}_{4}]|=
	|\mathcal{P}_{\mathbb{Z}_4}(\mathbb{Z}_{4}[\alpha])|.\]
	This shows that in general the embedding of Section~\ref{finitefieldcase} 
	need not be isomorphism.
\end{remark}
\noindent {\bf Acknowledgment.}  The author is supported by the Austrian Science Fund (FWF):  P 30934-N35.
\bibliographystyle{amsplain}
\bibliography{UVar}
%    Insert the bibliography data here.

\end{document}